\documentclass{sig-alternate-05-2015}

\usepackage[utf8]{inputenc}

\usepackage{amsmath,amssymb}
\usepackage{amsthmnoproof}
\usepackage{amsrefs}
\usepackage[usenames,dvipsnames]{color}
\usepackage{stmaryrd}
\usepackage{enumerate}
\usepackage[algoruled,vlined,english,linesnumbered]{algorithm2e}
\usepackage[pdfpagelabels,colorlinks=true,citecolor=blue]{hyperref}
\usepackage{comment}
\usepackage{multirow}
\usepackage{tikz}

\newcommand{\noopsort}[1]{}
\DeclareMathOperator{\NP}{NP}

\DeclareMathOperator{\val}{val}

\DeclareMathOperator{\lcm}{lcm}

\newcommand{\Z}{\mathbb Z}
\newcommand{\Zp}{\Z_p}
\newcommand{\Q}{\mathbb Q}
\newcommand{\Qp}{\Q_p}

\newcommand{\R}{\mathbb R}

\newcommand{\softO}{O\tilde{~}}

\newcommand{\NF}{\text{\rm NF}}

\renewcommand{\mod}{\,\%\,}
\renewcommand{\div}{\,/\hspace{-0.3em}/\,}
\newcommand{\nfop}[1]{%
\raisebox{-0.1mm}{\begin{tikzpicture}
\draw[transparent] (-0.18,0)--(0.18,0);
\draw[very thin] (0,0) circle (0.11cm);
\begin{scope}
\clip (0,0) circle (0.11cm);
\node[scale=0.75] at (0,0) { $#1$ };
\end{scope}
\end{tikzpicture}}}
\newcommand{\nfplus}{\nfop+}
\newcommand{\nftimes}{\nfop\times}
\newcommand{\nfmod}{\nfop\mod}
\newcommand{\nfdiv}{\nfop\div}

\newcommand{\app}{\textrm{app}}
\newcommand{\Epi}{\textrm{Epi}}
\renewcommand{\c}{\text{\rm c}}

\newcommand{\tinyplus}{\raisebox{0.4mm}{\hspace{0.4mm}\tiny +\hspace{0.4mm}}}

\definecolor{purple}{rgb}{0.6,0,0.6}

\permission{%
}

\begin{document}

\newtheorem{theo}{Theorem}[section]
\newtheorem{lem}[theo]{Lemma}
\newtheorem{prop}[theo]{Proposition}
\newtheorem{cor}[theo]{Corollary}
\newtheorem{quest}[theo]{Question}
\newtheorem{conj}[theo]{Conjecture}
\theoremstyle{definition}
\newtheorem{rem}[theo]{Remark}
\newtheorem{ex}[theo]{Example}
\newtheorem{deftn}[theo]{Definition}

\title{Division and Slope Factorization of p-Adic Polynomials}

\numberofauthors{3}
\author{
\alignauthor Xavier Caruso\\
  \affaddr{Universit\'e Rennes 1}\\
  \affaddr{\textsf{xavier.caruso@normalesup.org}}
\alignauthor David Roe \\
  \affaddr{University of Pittsburgh}\\
  \affaddr{\textsf{roed.math@gmail.com}}
\alignauthor Tristan Vaccon\\
  \affaddr{JSPS--Rikkyo University}\\
  \affaddr{\textsf{vaccon@rikkyo.ac.jp}}
}

\maketitle

\begin{abstract}
We study two important operations on polynomials defined over complete 
discrete valuation fields: Euclidean division and factorization. In particular,
we design a simple and efficient algorithm for computing slope factorizations,
based on Newton iteration.  One of its main features is that we avoid working
with fractional exponents.
We pay particular attention to stability, and analyze the behavior
of the algorithm using several precision models.
\end{abstract}


%
%

\section{Introduction}

Polynomial factorization is a fundamental problem in computational algebra.
The algorithms used to solve it depend on the ring of coefficients, with finite fields,
local fields, number fields and rings of integers of particular interest to number theorists.
In this article, we focus on a task that forms a building block for factorization
algorithms over complete discrete valuation fields: the decomposition into factors based on the
slopes of the Newton polygon.

The Newton polygon of a polynomial $f(X) = \sum a_i X^i$ over such a field is given by
the convex hull of the points $(i, \val(a_i))$ and the point $(0, +\infty)$.  The lower
boundary of this polygon consists of line segments $(x_j,y_j)$ -- $(x_{j+1},y_{j+1})$
of slope $s_j$.  The slope factorization of $f(X)$ expresses $f(X)$ as a product
of polynomials $g_j(X)$ with degree $x_{j+1} - x_j$ whose roots all have valuation $-s_j$.
Our main result is a new algorithm for computing these $g_j(X)$.

Polynomial factorization over local fields has seen a great deal of progress recently
\citelist{\cite{pauli:10a} \cite{guardia-nart-pauli:12a} \cite{guardia-montes-nart:08a} \cite{montes:99a}}
following an algorithm of Montes.  Slope factorization provides a subroutine
in such algorithms \cite[Section~2]{pauli:10a}.  For the most difficult inputs, it is not the dominant
contributor to the runtime of the algorithm, but in some circumstances it will be.
We underline moreover that the methods introduced in this paper extend 
partially to the noncommutative setting and appear this way as an
essential building block in several decomposition algorithms of 
$p$-adic Galois representations and $p$-adic differential 
equations~\cite{caruso:16}.

Any computation with $p$-adic fields must work with approximations modulo finite
powers of $p$, and one of the key requirements in designing an algorithm is
an analysis of how the precision of the variables evolve over the computation.
We work with precision models developed by the same authors
\cite[Section~4.2]{caruso-roe-vaccon:14a}, focusing on the lattice and Newton models.
As part of the analysis of the slope factorization algorithm, we describe
how the precision of the quotient and remainder depend on the input polynomials
in Euclidean division.

\medskip

\noindent
{\bf Main Results.}
Suppose that the Newton polygon of $P(X) = \sum_{i=0}^n a_i X^i$
has a break at $i=d$.  Set $A_0 = \sum_{i=0}^d a_i X^i$, $V_0 = 1$ and
\begin{align*}
A_{i+1} &= A_i + (V_i P \mod A_i) \\
B_{i+1} &= P \div A_{i+1} \\
V_{i+1} &= (2V_i - V_i^2B_{i+1}) \mod A_{i+1}.
\end{align*}
Our main result is Theorem \ref{theo:slope-factor},
which states that the sequence $(A_i)$ converges quadratically
to a divisor of $P$. This provides a quasi-optimal simple-to-implement 
algorithm for computing slope factorizations.
We moreover carry out a careful study of the precision and, applying
a strategy coming from~\cite{caruso-roe-vaccon:14a}, we end up with
an algorithm that outputs optimal results regarding to accuracy.

In order to prove Theorem~\ref{theo:slope-factor}, we also determine
the precision of the quotient and remainder in Euclidean division,
which may be of independent interest.  These results are found in
Section \ref{sec:prec_track}.

\medskip

\noindent
{\bf Organization of the paper.}
After setting notation, in
Section \ref{sec:prec_data} we recall various models for tracking
precision in polynomial arithmetic.  We give some background
on Newton polygons and explain how using lattices to store
precision can allow for extra \emph{diffuse} $p$-adic digits that
are not localized on any single coefficient.

In Section \ref{sec:quo_rem}, we consider Euclidean division.
We describe in Theorem \ref{theo:EDivisionNP}
how the Newton polygons of the quotient and remainder
depend on numerator and denominator.  We use this
result to describe in Proposition \ref{prop:NewtonprecEuclide}
the precision evolution in Euclidean division using
the Newton precision model.  We then compare the
precision performance of Euclidean division in the
jagged, Newton and lattice models experimentally,
finding different behavior depending on the modulus.

Finally, in Section \ref{sec:slope_fac} we describe
our slope factorization algorithm, which is based on
a Newton iteration.  Unlike other algorithms for
slope factorization, ours does not require working
with fractional exponents. 
In Theorem \ref{theo:slope-factor} we define a sequence
of polynomials that will converge to the factors
determined by an extremal point in the Newton
polygon.  We then discuss the precision behavior
of the algorithm.

\medskip

\noindent
{\bf Notations.}
Throughout this paper, we fix a complete discrete valuation field $K$; 
we denote by $\val : K \to \Z \cup \{+\infty\}$ the valuation on it and 
by $W$ its ring of integers (\emph{i.e.} the set of elements with 
nonnegative valuation). We assume that $\val$ is normalized so that it
is surjective and denote by $\pi$ a uniformizer of $K$, that is an 
element of valuation $1$. Denoting by $S \subset W$ a fixed set of 
representatives of the classes modulo $\pi$ and assuming $0 \in S$, 
one can prove that each element in $x \in K$ can be represented 
uniquely as a convergent series:
\begin{equation}
\label{eq:CDVFseries}
x = \sum_{i=\val(x)}^{+\infty} a_i \pi^i
\quad \text{with} \quad a_i \in S.
\end{equation}
The two most important examples are the field of $p$-adic numbers $K = 
\Qp$ and the field of Laurent series $K = k((t))$ over a field $k$. The 
valuation on them are the $p$-adic valuation and the usual valuation of 
a Laurent series respectively. Their ring of integers are therefore 
$\Zp$ and $k[[t]]$ respectively. A distinguished uniformizer is $p$ 
and $t$ whereas a possible set $S$ is $\{0, \ldots, p-1\}$ and $k$
respectively.
The reader who is not familiar with complete discrete valuation fields
may assume (without sacrifying too much to the generality) that $K$ is
one of the two aforementioned examples.

In what follows, the notation $K[X]$ refers to the ring of univariate 
polynomials with coefficients in $K$. The subspace of polynomials of 
degree at most $n$ (resp. exactly $n$) is denoted by $K_{\leq n}[X]$ 
(resp. $K_{=n}[X]$).

\section{Precision data} \label{sec:prec_data}

Elements in $K$ (and \emph{a fortiori} in $K[X]$) carry an infinite 
amount of information. They thus cannot be stored entirely in the
memory of a computer and have to be truncated. Elements of $K$ are
usually represented by truncating Eq.\eqref{eq:CDVFseries} as
follows:
\begin{equation}
\label{eq:CDVFseriesO}
x = \sum_{i=v}^{N-1} a_i \pi^i + O(\pi^N)
\end{equation}
where $N$ is an integer called the \emph{absolute precision} and 
the notation $O(\pi^N)$ means that the coefficients $a_i$ for $i
\geq N$ are discarded. If $N > v$ and $a_v \neq 0$, the integer $v$ 
is the valuation of $x$ and the difference $N-v$ is called the
\emph{relative precision}.
Alternatively, one may think that the writing~\eqref{eq:CDVFseriesO}
represents a subset of $K$ which consists of all elements in $K$ for
which the $a_i$'s in the range $[v,N-1]$ are those specified. From the
metric point of view, this is a ball (centered at any point inside it).

It is worth noting that tracking precision using this representation is 
rather easy. For example, if $x$ and $y$ are known with absolute (resp. 
relative) precision $N_x$ and $N_y$ respectively, one can compute the 
sum $x+y$ (resp. the product $xy$) at absolute (resp. relative) 
precision $\min(N_x,N_y)$. Computations with $p$-adic and Laurent
series are often handled this way on symbolic computation softwares.

\subsection{Precision for polynomials}

The situation is much more subtle when we are working with a collection 
of elements of $K$ (\emph{e.g.} a polynomial) and not just a single one.
Indeed, several precision data may be considered and, as we shall see
later, each of them has its own interest. Below we detail three models 
of precision for the special case of polynomials.

\medskip

\noindent
{\bf Flat precision.}
The simplest method for tracking the precision of a polynomial is
to record each coefficient modulo a fixed power of $p$.  While
easy to analyze and implement, this method suffers when
applied to polynomials whose Newton polygons are far from flat.

\medskip

\noindent
{\bf Jagged precision.}
The next obvious approach is to record the precision of each
coefficient individually, a method that we will refer to as \emph{jagged}
precision.  Jagged precision is commonly implemented
in computer algebra systems, since standard polynomial algorithms
can be written for generic coefficient rings.  However, these
generic implementations often have suboptimal precision behavior,
since combining intermediate expressions into a final answer
may lose precision.  Moreover, when compared to the Newton
precision model, extra precision in the middle coefficients, above
the Newton polygon of the remaining terms, will have no effect
on any of the values of that polynomial.

\medskip

\noindent
{\bf Newton precision.} 
We now move to \emph{Newton precision} data. They can be actually seen 
as particular instances of jagged precision but there exist for them 
better representations and better algorithms.

\begin{deftn}
A \emph{Newton function of degree $n$} is a convex function 
$\varphi : [0,n] \to \R \cup \{+\infty\}$ which is piecewise affine, 
which takes a finite value at $n$ and whose epigraph $\Epi(\varphi)$ 
have extremal points with integral abscissa.
\end{deftn}

\begin{rem}
The datum of $\varphi$ is equivalent to that of $\Epi(\varphi)$ and they 
can easily be represented and manipulated on a computer.
\end{rem}

We recall that one can attach a Newton function to each polynomial.
If $P(X) = \sum_{i=0}^n a_n X^n \in K_n[X]$, we define its Newton
polygon $\NP(P)$ as the convex hull of the points $(i,\val(a_i))$ 
($1 \leq i \leq n$) together with the point at infinity $(0,+\infty)$
and then its Newton function $\NF(P) : [0,n] \to \R$ as the unique
function whose epigraph is $\NP(P)$. It is well known \cite[Section~1.6]{dwork-geratto-sullivan:Gfunctions} that:
$$\begin{array}{r@{\,\,}c@{\,\,}l}
\NP(P+Q) & \subset & \text{Conv}\big(\NP(P) \cup \NP(Q)\big) \smallskip \\
\NP(PQ) & = & \NP(P) + \NP(Q)
\end{array}$$
where $\text{Conv}$ denotes the convex hull and the plus sign stands 
for the Minkowski sum. This translates to:
$$\begin{array}{r@{\,\,}c@{\,\,}l}
\NF(P+Q) \geq \NF(P) \nfplus \NF(Q) \smallskip \\
\NF(PQ) = \NF(P) \nftimes \NF(Q)
\end{array}$$
where the operations $\nfplus$ and $\nftimes$ are defined accordingly.
There exist classical algorithms for computing these two operations
whose complexity is quasi-linear with respect to the degree.

In a similar fashion, Newton functions can be used to model precision: 
given a Newton function $\varphi$ of degree $n$, we agree that a polynomial 
of degree at most $n$ is given at precision $O(\varphi)$ when, for all $i$,
its $i$-th coefficient is given at precision $O\big(\pi^{\lceil \varphi(i)
\rceil}\big)$ (where $\lceil \cdot \rceil$ is the ceiling function).
In the sequel, we shall write
$O(\varphi) = \sum_{i=0}^n O\big(\pi^{\lceil \varphi(i) \rceil}\big) \cdot X^i$
and use the notation $\sum_{i=0}^n a_i X^i + O(\varphi)$ (where the
coefficients $a_i$ are given by truncated series) to refer to a 
polynomial given at precision $O(\varphi)$.

It is easily checked that if $P$ and $Q$ are two polynomials known at 
precision $O(\varphi_P)$ and $O(\varphi_Q)$ respectively, then $P+Q$ is 
known at precision $O(\varphi_P \nfplus \varphi_Q)$ and $PQ$ is known at 
precision $O\big((\varphi_P \nftimes \NF(Q)) \nfplus (\NF(P) \nftimes 
\varphi_Q)\big)$.

\begin{deftn}
\label{def:nondeg}
Let $P = P_\app + O(\varphi_P)$. We say that the Newton precision 
$O(\varphi_P)$ on $P$ is \emph{nondegenerate} if $\varphi_P \geq 
\NF(P_\app)$ and $\varphi_P(x) > y$ for all extremal point $(x,y)$ of 
$\NP(P_\app)$.
\end{deftn}

We notice that, under the conditions of the above definition, the
Newton polygon of $P$ is well defined. Indeed, if $\delta P$ is any
polynomial whose Newton function is not less than $\varphi_P$, we
have $\NP(P_\app + \delta P) = \NP(P_\app)$.

\medskip

\noindent
{\bf Lattice precision.}
The notion of \emph{lattice precision} was developed in 
\cite{caruso-roe-vaccon:14a}. It encompasses the two previous models and has
the decisive advantage of precision optimality.
 As a counterpart, it might be very space-consuming and time-consuming 
 for polynomials of large degree.

\begin{deftn}
Let $V$ be a finite dimensional vector space over $K$. A lattice
in $V$ is a sub-$W$-module of $V$ generated by a $K$-basis of
$V$.
\end{deftn}

\noindent
We fix an integer $n$. A lattice precision datum for a polynomial of 
degree $n$ is a lattice $H$ lying in the vector space $K_{\leq n}[X]$. 
We shall sometimes denote it $O(H)$ in order to emphasize that it should 
be considered as a precision datum. The notation $P_\app(X) + O(H)$ then 
refers to any polynomial in the $W$-affine space $P_\app(X) + H$. 
Tracking lattice precision can be done using differentials as shown in
\cite[Lemma~3.4 and Proposition~3.12]{caruso-roe-vaccon:14a}: if $f : K_{\leq n}[X] 
\to K_{\leq m}[X]$ denotes any strictly differentiable function with
surjective differential, under mild assumption on $H$, we have:
$$f(P_\app(X)+H) = f(P_\app(X)) + f'(P_\app(X))(H)$$
where $f'(P_\app(X))$ denotes the differential of $f$ at $P_\app(X)$. 
The equality sign reflets the optimality of the method.

As already mentioned, the jagged precision model is a particular case of 
the lattice precision. Indeed, a precision of the shape $\sum_{i=0}^n 
O(\pi^{N_i}) X^i$ corresponds to the lattice generated by the elements 
$\pi^{N_i} X^i$ ($0 \leq i \leq n$). This remark is the origin of the 
notion of \emph{diffused digits of precision} introduced in 
\cite[Definition~2.3]{caruso-roe-vaccon:15a}. We shall use it repeatedly in the 
sequel in order to compare the behaviour of the three aforementioned 
precision data in concrete situations.

\section{Euclidean division} \label{sec:quo_rem}

Euclidean division provides a building block for many algorithms associated
to polynomials in one variable.  In order to analyze the precision behavior
of such algorithms, we need to first understand the precision attached to
the quotient and remainder when dividing two polynomials.
In the sequel, we use the notation $A \div B$ and $A \mod B$
for the polynomials satisfying $A = (A \div B) \cdot B + (A \mod B)$
and $\deg(A \mod B) < \deg(B)$.

\subsection{Euclidean division of Newton functions}

\begin{deftn}
Let $\varphi$ and $\psi$ be two Newton functions of degree $n$ and
$d$ respectively. Set $\lambda = \psi(d) - \psi(d-1)$. Letting $\Delta$
be the greatest affine function of slope $\lambda$ with $\Delta \leq
\varphi_{|[d,n]}$ and $\delta = \Delta(d) - \psi(d)$, we define:
$$\begin{array}{r@{\,\,}rcl}
\varphi \nfmod \psi & 
\multicolumn{3}{@{}l}{= \varphi_{|[0,d{-}1]} \nfplus \big( 
\psi_{|[0,d{-}1]} + \delta\big)} \medskip \\
\varphi \nfdiv \psi & : \, [0, n-d] & \to & \R \cup \{+\infty\} \smallskip \\
& x & \mapsto & \inf_{h \geq 0} \varphi(x+d+h) - \lambda h.
\end{array}$$
\end{deftn}

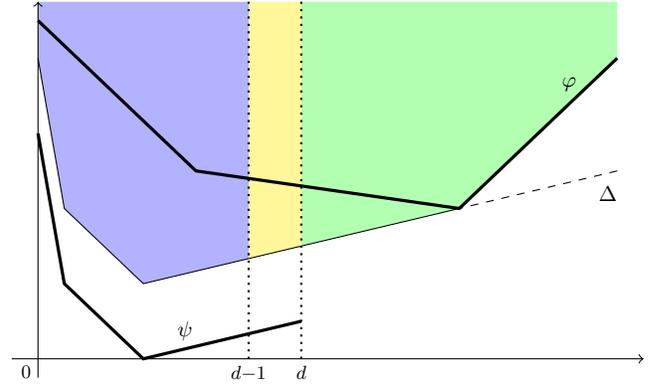
\begin{figure}
\hfill
\begin{tikzpicture}[xscale=0.7, yscale=0.5]
\fill[blue!30] (4,9.5)--(4,2.667)--(2,2)--(0.5,4)--(0,8)--(0,9.5)--cycle;
\fill[yellow!50] (5,9.5)--(5,3)--(4,2.667)--(4,9.5)--cycle;
\fill[green!30] (5,9.5)--(5,3)--(8,4)--(11,8)--(11,9.5)--cycle;
\draw[->] (-0.5,0)--(11.5,0);
\draw[->] (0,-0.5)--(0,9.5);
\draw[very thick] (0,6)--(0.5,2)--(2,0)--(5,1);
\draw[very thick] (0,9)--(3,5)--(8,4)--(11,8);
\draw[dashed] (11,5)--(8,4);
\draw (8,4)--(2,2)--(0.5,4)--(0,8);
\draw[dotted,thick] (5,0)--(5,9.5);
\draw[dotted,thick] (4,0)--(4,9.5);
\node[scale=0.8, below left] at (0,0) { $0$ };
\node[scale=0.8, below] at (5,0) { $d$ };
\node[scale=0.8, below] at (4,0) { $d{-}1$ };
\node[scale=0.9, below right] at (2.5,1.2) { $\psi$ };
\node[scale=0.9, right] at (9.8,7.3) { $\varphi$ };
\node[scale=0.9, right] at (10.5,4.4) { $\Delta$ };
\end{tikzpicture}
\hfill \null

\vspace{-6mm}

\caption{Euclidean division of Newton functions}
\label{fig:NewtonEuclide}
\end{figure}

Figure~\ref{fig:NewtonEuclide} illustrates the definition: if $\varphi$ 
and $\psi$ are the functions represented on the diagram, the epigraph of 
$\varphi \nfmod \psi$ is the blue area whereas that of $\varphi \nfdiv \psi$ 
is the green area translated by $(-d,0)$.
It is an easy exercise (left to the reader) to design quasi-linear
algorithms for computing $\varphi \nfmod \psi$ and $\varphi \nfdiv \psi$.

\begin{theo}
\label{theo:EDivisionNP}
Given $A, B \in K[X]$ with $B \neq 0$, we have:
\begin{align}
\NF(A \mod B) & \geq \NF(A) \nfmod \NF(B) \label{eq:AmodB} \\
\text{and} \hspace{4mm}
\NF(A \div B) & \geq \NF(A) \nfdiv \NF(B) \label{eq:AdivB}
\end{align}
\end{theo}

\begin{proof}
Write $A = A_{<d} + A_{\geq d}$ where $A_{<d}$ (resp. $A_{\geq d}$)
consists of monomials of $A$ of degree less than $d$ (resp. at least
$d$). Noting that:
$$A \mod B = A_{<d} + (A_{\geq d} \mod B)
\quad \text{and} \quad
A \div B = A_{\geq d} \div B$$
we may assume that $A = A_{\geq d}$.

Let us now prove Eq.~\eqref{eq:AmodB}. 
Replacing $B$ by $c^{-1} B$ where $c$ denotes the leading coefficient
of $B$, we may assume that $B$ is monic. Using linearity, we may further
assume that $A$ is a monomial. Set $R_n = X^n \mod B$. The relation
we have to prove is:
$$\NF(R_n)(x) \geq \NF(B)(x) - \lambda(n-d)
\quad \text{for } n\geq d.$$
We proceed by induction. The initialisation is clear because $R_d$
agrees with $(-B)$ up to degree $d{-}1$. We have the relation
$R_{n+1} = X R_n - c_n B$
where $c_n$ is the coefficient in $X^{d-1}$ of $R_n$. Thanks to the
induction hypothesis, we have:
\begin{align*}
\val(c_n) \geq \NF(R_n)(d{-}1) & \geq \NF(B)(d{-}1) - \lambda(n{-}d) \\
& = - \lambda(n{+}1{-}d)
\end{align*}
since $\lambda = -\NF(B)(d{-}1)$ because $B$ is monic. Therefore
$\NF(c_n B)(x) \geq \NF(B)(x) - \lambda(n{+}1{-}d)$ for all $x$. On the 
other hand, for all $x$, we have:
$$\NF(X R_n)(x) = \NF(R_n)(x{-}1) \geq \NF(R_n)(x) - \lambda$$
from what we get $\NF(X R_n)(x) \geq \NF(B)(x) - \lambda(n{+}1{-}d)$. As a
consequence
$\NF(R_{n+1})(x) \geq \NF(B)(x) - \lambda(n{+}1{-}d)$
and the induction follows.
Eq.~\eqref{eq:AdivB} is now derived from:
$$\NF(A \div B) \nftimes \NF(B) \geq \NF(A) \nfplus \NF(A \mod B)$$
using the estimation on $\NF(A\mod B)$ we have just proved
(see Figure~\ref{fig:NewtonEuclide}).
\end{proof}

\subsection{Tracking precision} \label{sec:prec_track}

\medskip

\noindent
{\bf Newton precision.}
We first analyze the precision behavior of Euclidean division
in the Newton model.  Concretely, we pick $A, B \in K[X]$ two polynomials which 
are known at precision $O(\varphi_A)$ and $O(\varphi_B)$ respectively:
$$A = A_\app + O(\varphi_A)
\quad \text{and} \quad
B = B_\app + O(\varphi_B)$$
Here $A_\app$ and $B_\app$ are some approximations of $A$ and $B$ respectively 
and $\varphi_A$ and $\varphi_B$ denotes two Newton functions of degree 
$\deg A$ and $\deg B$ respectively.
We are interested in determining the precision on $A \mod B$ and $A
\div B$. The following proposition gives a theoretical answer under
mild assumptions.

\begin{prop}
\label{prop:NewtonprecEuclide}
We keep the above notations and assume that the Newton precisions
$O(\varphi_A)$ and $O(\varphi_B)$ on $A$ and $B$ respectively are 
both nondegenerate (\emph{cf} Definition~\ref{def:nondeg}). Then,
setting:
$$\varphi = \varphi_A \nfplus 
\big[ \varphi_B \nftimes \big(\NF(A) \nfdiv \NF(B)\big) \big]$$
the polynomials $A \div B$ and $A \mod B$ are known at precision 
$O(\varphi \nfdiv \NF(B))$ and $O(\varphi \nfmod \NF(B))$ respectively.
\end{prop}

\begin{proof}
Let $\delta A$ (resp. $\delta B$) be a polynomial whose Newton 
function is not less than $\varphi_A$ (resp. $\varphi_B$) and define
$\delta Q$ and $\delta R$ by:
\begin{align*}
Q_\app + \delta Q & = (A_\app + \delta A) \div (B_\app + \delta B) \\
R_\app + \delta R & = (A_\app + \delta A) \mod (B_\app + \delta B)
\end{align*}
where $Q_\app = A_\app \div B_\app$ and $R_\app = A_\app \mod B_\app$.
We have to show that $\NF(\delta Q) \geq \varphi \nfdiv \NF(B)$ and
$\NF(\delta R) \geq \varphi \nfmod \NF(B)$.
Set $\delta X = \delta A - Q_\app \delta B$. 
Using Theorem~\ref{theo:EDivisionNP}, we obtain $\NF(Q_\app) \geq
\NF(A) \nfdiv \NF(B)$ and consequently $\NF(\delta X) \geq \varphi$.
On the other hand, an easy computation yields
$$\delta X = (B + \delta B) \cdot \delta Q + \delta R$$
so that $\delta Q = \delta X \div (B + \delta B)$ and 
$\delta R = \delta X \mod (B + \delta B)$. Using again 
Theorem~\ref{theo:EDivisionNP}, we get the desired result.
\end{proof}

With this result in hand, we may split the computation of
Euclidean division into two pieces, first computing approximations
$Q_\app$ and $R_\app$ and separately computing $\delta Q$ and $\delta R$.
Both the approximations and the precision can be computing in time that is
quasi-linear in the degree. 

\medskip

\noindent
{\bf Lattice precision.}
We now move to lattice precision. We pick $A$ and $B$ two polynomials 
of respective degree $n$ and $d$ and assume that they are known at 
precision $O(H_A)$ and $O(H_B)$ respectively:
$$A = A_\app + O(H_A)
\quad \text{and} \quad
B = B_\app + O(H_B).$$
where $H_A \in K_{\leq n}[X]$ and $H_B \in K_{\leq d}[X]$ are lattices. 
According to the results of \cite{caruso-roe-vaccon:14a}, in order to determine the 
precision on $A \div B$ and $A \mod B$, we need to compute the 
differential of the mappings $(X,Y) \mapsto X \div Y$ and $(X,Y)
\mapsto X \mod Y$ at the point $(A_\app, B_\app)$. Writing $Q_\app = A_\app \div
B_\app$ and $R_\app = A_\app \mod B_\app$, this can be done by
expanding the relation:
$$A_\app + dA = (B_\app + dB) (Q_\app + dQ) + (R_\app + dR)$$
and neglecting the terms of order $\geq 2$. We get this way 
$dA = B_\app dQ + Q_\app dB + dR$
meaning that $dQ$ and $dR$ appears respectively as the quotient and
the remainder of the Euclidean division of $dX = dA - Q_\app dB$ by $B_\app$.

Once this has been done, the strategy is quite similar to that explained 
for Newton precision: compute approximations and precision lattices
separately for quotient and remainder.

\subsection{An example: modular multiplication}

For this example, we work over $W = \Z_2$ and fix a monic polynomial $M 
\in \Z[X]$ (known exactly) of degree $5$. Our aim is to compare the 
numerical stability of the multiplication in the quotient $\Z_2[X]/M$ 
depending on the precision model we are using. In order to do so, we 
pick $n$ random polynomials $P_1, \ldots, P_n$ in $\Z_2[X]/M(X)$ 
(according to the Haar measure) whose coefficients are all known at 
precision $O(2^N)$ for some large integer $N$. We then compute the 
product of the $P_i$'s using the following quite naive algorithm.

\noindent\hrulefill

%

\noindent 1.\ {\bf set} $P = 1$

\noindent 2.\ {\bf for} $i=1,\dots,n$ {\bf do} {\bf compute} $P = (P 
\cdot P_i) \mod M$

\noindent 3.\ {\bf return} $P$

\vspace{-1ex}\noindent\hrulefill

\medskip

The table of Figure~\ref{fig:modularmult} reports the average gain of 
\emph{absolute} precision $G$ which is observed while executing the 
algorithm above for various modulus and $n$. The average is taken on a 
sample of $1000$ random inputs. We recall that $G$ is defined as 
follows:

\noindent $\bullet$
in the case of jagged and Newton precision, the precision on the output 
may be written into the form $\sum_{i=0}^4 O(2^{N_i}) \: X^i$ and 
$G = \sum_{i=0}^4 (N_i - N)$;

\noindent $\bullet$
in the case of lattice precision, the precision on the output is a 
lattice $H$ and $G$ is the index of $H$ in $2^N \mathcal L$ where 
$\mathcal L = \Z_2[X]/M$ is the standard lattice; in that case, we write 
$G$ as a sum $G_{\text{nd}} + G_{\text{d}}$ where $G_{\text{d}}$ is the 
index of $H$ in the largest lattice $H_0$ contained in $H$ which can be 
generated by elements of the shape $2^{N_i} X^i$ ($0 \leq i \leq 4$). 
(The term $G_d$ corresponds to diffused digits according to 
\cite[Definition~2.3]{caruso-roe-vaccon:15a}.)

\begin{figure}
{\small%
\noindent \hfill%
\renewcommand{\arraystretch}{1.2}%
\begin{tabular}{|c|c|c|c|c|}%
\hline
\multirow{3}{*}{Modulus $M$} & 
\multirow{3}{*}{\hspace{2mm}$n$\hspace{2mm}} & 
\multicolumn{3}{c|}{Gain of precision} \\
\cline{3-5}
& & \multirow{2}{*}{Jagged} & \multirow{2}{*}{Newton}
      & \raisebox{-0.5mm}{Lattice} \\
& & & & \raisebox{0.5mm}{\scriptsize (not dif.\tinyplus dif.)} \\
\hline
\multirow{3}{*}{\begin{tabular}{@{}c@{}} $X^5 + X^2 + 1$ \\ {\scriptsize (Irred. mod $2$)} \end{tabular}}
& $10$ & $\phantom{00}0.2$ & $\phantom{00}0.2$ & $\phantom{00}0.2\tinyplus\phantom{00}0.0$ \\
& $50$ & $\phantom{00}4.2$ & $\phantom{00}4.2$ & $\phantom{00}4.2\tinyplus\phantom{00}0.0$ \\
& $100$ & $\phantom{0}11.2$ & $\phantom{0}11.2$ & $\phantom{0}11.2\tinyplus\phantom{00}0.0$ \\
\hline
\multirow{3}{*}{\begin{tabular}{@{}c@{}} $X^5 + 1$ \\ {\scriptsize (Sep. mod $2$)} \end{tabular}}
& $10$ & $\phantom{00}0.4$ & $\phantom{00}0.4$ & $\phantom{00}0.9\tinyplus\phantom{00}6.0$ \\
& $50$ & $\phantom{00}5.6$ & $\phantom{00}5.6$ & $\phantom{0}11.1\tinyplus\phantom{0}42.0$ \\
& $100$ & $\phantom{0}13.6$ & $\phantom{0}13.6$ & $\phantom{0}27.0\tinyplus\phantom{0}87.0$ \\
\hline
\multirow{3}{*}{\begin{tabular}{@{}c@{}} $X^5 + 2$ \\ {\scriptsize (Eisenstein)} \end{tabular}}
& $10$ & $\phantom{00}6.2$ & $\phantom{00}6.2$ & $\phantom{00}6.2\tinyplus\phantom{00}0.0$ \\
& $50$ & $\phantom{0}44.0$ & $\phantom{0}44.0$ & $\phantom{0}44.0\tinyplus\phantom{00}0.0$ \\
& $100$ & $\phantom{0}92.5$ & $\phantom{0}92.5$ & $\phantom{0}92.5\tinyplus\phantom{00}0.0$ \\
\hline
\multirow{3}{*}{\begin{tabular}{@{}c@{}} $(X+1)^5 + 2$ \\ {\scriptsize (Shift Eisenstein)} \end{tabular}}
& $10$ & $\phantom{00}0.6$ & $\phantom{00}0.6$ & $\phantom{00}4.7\tinyplus\phantom{00}1.4$ \\
& $50$ & $\phantom{00}7.1$ & $\phantom{00}7.1$ & $\phantom{0}42.6\tinyplus\phantom{00}1.4$ \\
& $100$ & $\phantom{0}15.1$ & $\phantom{0}15.1$ & $\phantom{0}91.8\tinyplus\phantom{00}1.4$ \\
\hline
\multirow{3}{*}{\begin{tabular}{@{}c@{}} $X^5 + X + 2$ \\ {\scriptsize (Two slopes)} \end{tabular}}
& $10$ & $\phantom{00}1.7$ & $\phantom{00}1.7$ & $\phantom{00}7.9\tinyplus\phantom{00}9.8$ \\
& $50$ & $\phantom{00}8.1$ & $\phantom{00}8.1$ & $\phantom{0}70.7\tinyplus\phantom{0}59.8$ \\
& $100$ & $\phantom{0}16.1$ & $\phantom{0}16.1$ & $152.6\tinyplus125.9$ \\
\hline
\end{tabular}
\hfill \null}

\caption{Precision for modular multiplication}
\label{fig:modularmult}
\end{figure}

We observe several interesting properties. First of all, the gains for 
Newton precision and jagged precision always agree though one may have 
thought at first that Newton precision is weaker. Since performing
precision computations in the Newton framework is cheaper, it seems
(at least on this example) that using the jagged precision model is not
relevant.

On the other hand, the lattice precision may end up with better results. 
Nevertheless this strongly depends on the modulus $M$. For instance, 
when $M$ is irreducible modulo $p=2$ or Eiseistein, there is apparently 
no benefit to using the lattice precision model. We emphasize that 
these two particular cases correspond to modulus that are usually used 
to define (unramified and totally ramified respectively) extensions of 
$\Q_2$.

For other moduli, the situation is quite different and the benefit of 
using the lattice precision model becomes more apparent. 
The comparison between the gain of precision in the jagged model and 
the number of not diffused digits in the lattice model makes sense:
indeed the latter appears as a theoretical upper bound of the former
and the difference between them quantifies the quality of the way we
track precision in the jagged (or the Newton) precision model. We 
observe that this difference is usually not negligible (\emph{cf} 
notably the case of $M(X) = (X+1)^5 + 2$) meaning that this quality 
is not very good in general.
As for diffused digits, they correspond to digits that cannot be 
``seen'' in the jagged precision model. Their number then measures the 
intrinsic limitations of this model. We observe that it can be very 
important as well in several cases.

The modulus $(X+1)^5 + 2$ shows the advantage of working
with lattice precision in intermediate computations.  Indeed,
the precision behavior using the lattice model closely parallels
that of $X^5+2$, since the lattices are related by a change of
variables.  But this structure is not detected in the Newton
or jagged models.

\section{Slope factorization} \label{sec:slope_fac}

A well-known theorem \cite[Theorem~6.1]{dwork-geratto-sullivan:Gfunctions} asserts that each 
extremal point $M$ in the Newton polygon $\NP(P)$ of a polynomial $P \in 
K[X]$ corresponds to a factorization $P = AB$ where the Newton polygon 
of $A$ (resp. $B$) is given by the part of $\NP(P)$ located at the left 
(resp. the right) of $M$. Such a factorization is often called a 
\emph{slope factorization}.

The aim of this section is to design efficient and stable algorithms 
for computing these factorizations. Precisely the algorithm we obtain 
has a quasi-optimal complexity (compared to the size of the input 
polynomial) and outputs a result whose precision is (close to be) 
optimal. Two of its important additional features are simplicity and 
flexibility.

\subsection{A Newton iteration}
\label{ssec:Newtoniter}

The factor $A$ defined above is usually obtained \emph{via} a Newton 
iteration after having prepared our polynomial by flattening the first 
slope using a change of variables involving possibly rational 
exponents. We introduce here a variant of this iteration which does not 
require the flattening step and is entirely defined over $K[X]$.

\begin{theo} \label{theo:slope-factor}
Let $P(X) = \sum_{i=0}^n a_i X^i$ be a polynomial of degree $n$ with
coefficients in $K$. We assume that $\NP(P)$ has an extremal point 
whose abscissa is $d$.
We define the sequences $(A_i)_{i \geq 0}$ and $(V_i)_{i \geq 0}$
recursively by:
\begin{align*}
A_0 & = \sum_{i=0}^d a_i X^i, \quad V_{0} = 1 \medskip \\
A_{i+1} &= A_i + (V_i P \:\mod\: A_i), \smallskip \\
V_{i+1} &= (2 V_i -V_i^2 B_{i+1} ) \mod A_{i+1} \\
& \hspace{2cm}\text{where } B_{i+1} = P \div A_{i+1}.
\end{align*}
Then the sequence $(A_i)$ converges to a divisor $A_\infty$ of $P$ 
of degree $d$ whose leading coefficient is $a_d$ and whose Newton 
function agrees with $\NF(P)$ on $[0,d]$.
Moreover, setting:
$$\kappa = \NF(P)(d{+}1) + \NF(P)(d{-}1) - 2 \cdot \NF(P)(d)$$
(with $\NF(P)(-1) = \NF(P)(n{+}1) = +\infty$ if necessary),
we have $\kappa > 0$ and the following rate of convergence:
\begin{equation}
\label{eq:rateconv}
\forall i \geq 0, \quad
\NF(A_\infty - A_i) \geq \NF(P)_{|[0,d{-}1]} \,+\, 2^i \kappa.
\end{equation}
\end{theo}


%

\begin{rem}
\label{rem:unicity}
The divisor $A$ is uniquely determined by the conditions of 
Theorem~\ref{theo:slope-factor}. Indeed, consider two divisors
$A$ and $A'$ of $P$ such that $\NF(A) = \NF(A') = \NF(P)_{|[0,d]}$.
Then $L = \lcm(A,A')$ is a divisor of $P$ as well and the 
slopes of its Newton polygon are all at most $\lambda_0 = \NF(P)(d) - 
\NF(P)(d{-}1)$. Therefore $\deg L = d$ and $L$ differs from $A$ and
$A'$ by a multiplicative nonzero constant. Then, if $A$ and $A'$ 
share in addition the same leading coefficient, they must coincide.
\end{rem}

The rest of this subsection is devoted to the proof of the theorem.
If $d = n$ (resp. $d = 0$), the sequence $A_i$ is constant equal to
$P$ (resp. to the constant coefficient of $P$) and theorem is clear.
We then assume $0 < d < n$. We set:
\begin{align*}
\lambda_0 & = \NF(P)(d) - \NF(P)(d{-}1) \\
\lambda_1 & = \NF(P)(d{+}1) - \NF(P)(d),
\end{align*}
so that $\kappa = \lambda_1 - \lambda_0$.
The existence of an extremal point of $\NP(P)$ located at abscissa
$d$ ensures that $\lambda_1 > \lambda_0$, \emph{i.e.} $\kappa > 0$.
For all indices $i$, we define:
$$\begin{array}{r@{\,\,}lr@{\,\,}l}
Q_i &= V_i P \div A_i, &
R_i & = V_i P \mod A_i = A_{i+1} - A_i, 
\smallskip \\
S_i &= P \mod A_i, &
T_i &= (1 - V_i B_i) \mod A_i
\end{array}$$
and when $\square$ is some letter, we put $\Delta \square_i = \square_{i+1} 
- \square_i$.

\begin{lem} \label{lem:formulae}
The following relations hold:
\begin{align}
\Delta B_i &= -(R_i B_{i+1}) \div A_i, \label{eqdef:Biminus} \\
\Delta S_i &= -(R_i B_{i+1}) \mod A_i,  \label{eqdef:Siminus} \\
S_{i} &= (B_i R_i + T_i S_{i-1} + T_i \: \Delta S_{i-1}) \mod A_{i}, \label{eqdef:Si2}  \\
\Delta V_i &= (V_i T_i - V_i^2 \: \Delta B_i) \mod A_{i}, \label{eqdef:Viminus}\\
1-Q_i &= T_i - (V_i S_i) \div A_i, \label{eqdef:Qiminus}\\
R_{i+1} &= (\Delta V_i \: S_{i+1}+(1{-}Q_i)R_i) \mod A_{i+1}, \label{eq:Riplus}\\
T_{i+1} &= (T_i + V_i \: \Delta B_i)^2 \mod{A_{i+1}}. \label{eq:ViBiplus}
\end{align}
\end{lem}

\begin{proof}
From $P= A_i B_i + S_i= A_{i+1} B_{i+1} + S_{i+1}$, we get
$- R_i B_{i+1}= \Delta B_i \cdot A_i + \Delta S_i$.
Hence, by consideration of degree, we obtain \eqref{eqdef:Biminus} 
and \eqref{eqdef:Siminus}.
On the other hand, from $V_i P = A_i Q_i + R_i=V_i (A_i B_i + S_i)$, we 
derive
\begin{equation}
\label{eq:ViBiQi} 
(V_i B_i -Q_i) \cdot A_i=R_i-V_i S_i. 
\end{equation} 
Thus $R_i=V_i S_i \mod A_i$.
Hence $B_i R_i = (S_i - S_i T_i) \mod A_i$ and $S_i = B_i R_i + S_i T_i 
\mod A_i$, from which \eqref{eqdef:Si2} follows directly 
By definition of $V_i$, we get $\Delta V_i = V_i (1-V_i B_{i+1}) \mod A_{i+1}$ and consequently \eqref{eqdef:Viminus}.
We now write $1-Q_i= T_i + (V_i B_i - Q_i)$. Using \eqref{eq:ViBiQi} and
noting that $\deg R_i < \deg A_i = d$, we 
get \eqref{eqdef:Qiminus}.

We have $V_i P=A_i Q_i+R_i=(A_{i+1}-R_i)Q_i+R_i$ and $V_{i+1} P=A_{i+1} 
Q_{i+1}+R_{i+1}$. Thus:
\begin{align*}
R_{i+1} &= \Delta V_i \: P + (1-Q_i)R_i \\
 &= (\Delta V_i \: S_{i+1} + (1-Q_i)R_i) \mod A_{i+1},
\end{align*} 
and \eqref{eq:Riplus} is proved. Finally
\begin{align*}
T_{i+1} &\equiv 1-2V_i B_{i+1}+V_i^2 B_{i+1}^2 \pmod {A_{i+1}} \\
&\equiv (1-V_i B_{i+1})^2 \pmod {A_{i+1}} \\
&\equiv (T_i + V_i \: \Delta B_i)^2 \pmod {A_{i+1}}
\end{align*}
which concludes the proof.
\end{proof}

If $\lambda_0 = -\infty$ or $\lambda_1 = +\infty$, the sequence $(A_i)$
is constant and the theorem is obvious. We then assume that $\lambda_0$
and $\lambda_1$ are both finite.

We define the function $\varphi : \R^+ \to \R 
\cup \{+\infty\}$ by:
\begin{equation} \label{eqdef:functionphi}
\varphi(x) = \begin{cases} \NF(P)(x) &\mbox{ if $x \leq d$} \\
\lambda_0 (x - d) + \NF(P)(d) &\mbox{ if $x > d$} \end{cases}
\end{equation}
We notice that, when the polynomial $P$ is changed into $cP$ where $c$ is a 
nonzero constant lying in a finite extension of $K$, the $A_i$'s are all 
multiplied by $c$ as well whereas the $B_i$'s and the $V_i$'s remained 
unchanged. Therefore, the theorem holds for $P$ if and only if it holds 
for $cP$. As a consequence we may assume that $P$ is normalized so that 
$\NF(P)(d) = d \lambda_0$, \emph{i.e.} $\varphi(x) = \lambda_0 x$ for $x > d$.
For a polynomial $Q \in K[X]$ of degree $n$, we further define:
\begin{align}
b_\varphi(Q) & = \min_{x \in [0,n]} \NF(Q)(x) - \varphi(x) 
\label{eqdef:bphi} \\
\text{and} \quad
b_i(Q) & = \min_{x \in [0,n]} \NF(Q)(x) - \lambda_i x
\quad \text{for } i \in \{0,1\}.
\label{eqdef:b0}
\end{align}
Set also $b_\varphi(0) = b_0(0) = b_1(0) = +\infty$ by convention.
With the normalization of $P$ we chose above, we have $b_0(P)
= b_\varphi(P) = 0$ and $b_\varphi(Q) \leq b_0(Q)$ for all polynomial
$Q$. Similarly $b_1(Q) \leq b_0(Q)$ for all $Q$.

\begin{lem}
\label{lem:bphib0}
Let $b \in \{b_\varphi, b_0, b_1\}$. For $Q_1, Q_2 \in K[X]$:

\smallskip

a) $b(Q_1+Q_2) \geq \min \big(b(Q_1), b(Q_2)\big)$

\smallskip

b) $b(Q_1Q_2) \geq \min b(Q_1) + b(Q_2)$ 

\smallskip

c) $b_\varphi(Q_1 Q_2) \geq b_\varphi(Q_1) + b_0(Q_2)$

\smallskip

\noindent
For $Q, A \in K[X]$ with $\deg A = d$ and $\NF(A) =
\NF(P)_{|[0,d]}$:

\smallskip

d) $b(Q \mod A) \geq b(Q)$

\smallskip

e) $b_0(Q \div A) \geq b_\varphi(Q)$.
\end{lem}

\begin{proof}
\emph{a)} is clear.

We skip the proof of \emph{b)} which is similar to that of \emph{c)}.

Let $t_1$ (resp. $t_2$) be the translation of vector $(0, b_\varphi
(Q_1))$ (resp. $(0, b_0(Q_2))$. It follows from the definition of
$b_\varphi$ that $\NP(Q_1)$ is a subset of $t_1(\Epi(\varphi))$ where
$\Epi(\varphi)$ denotes the epigraph of $\varphi$. Similarly $\NP(Q_2)
\subset t_2(C)$ where $C$ is the convex cone generated by the vectors 
starting from $(0,0)$ to $(0,1)$ and $(1, \lambda_0)$. Thus
$$\NP(Q_1 Q_2) \subset t_2 \circ t_1 \big(\Epi(\varphi) + C\big) =
t_2 \circ t_1\big(\Epi(\varphi)\big)$$
and \emph{c)} follows.

Finally \emph{d)} and \emph{e)} follows from
Theorem~\ref{theo:EDivisionNP}.
\end{proof}

We are now going to prove by induction on $i$ the conjonction of all
equalities and inequalities below:
\begin{equation}
\label{eq:induction}
\begin{array}{l}
\NF(A_i) = \varphi_{|[0,d]}, \smallskip \\
b_1(V_i) \geq 0, \quad
b_\varphi(R_i) \geq 2^i \kappa \smallskip \\
b_\varphi(S_i) \geq 0, \quad
b_0(T_i) \geq 2^i \kappa.
\end{array}
\end{equation}

Noting that $A_0$ and $P$ agree up to degree $d$ and that $\NP(P)$ has 
an extremal point at abscissa $d$, we get $\NF(A_0) = \varphi_{|[0,d]}$.
Clearly $b_1(V_0) \geq 0$ since $V_0 = 1$. It follows from the 
definitions that $R_0 = S_0 = P \mod A_0 = (P{-}A_0) \mod A_0$. 
We remark that $P{-}A_0 = \sum_{i=d+1}^n a_i X^i.$
Using Theorem \ref{theo:EDivisionNP}, we obtain that
$b_\varphi(P{-}A_0) \geq \kappa$ and then $b_\varphi(R_0) = b_\varphi(S_0)
\geq \kappa$ (see Figure~\ref{fig:boundR0S0}).
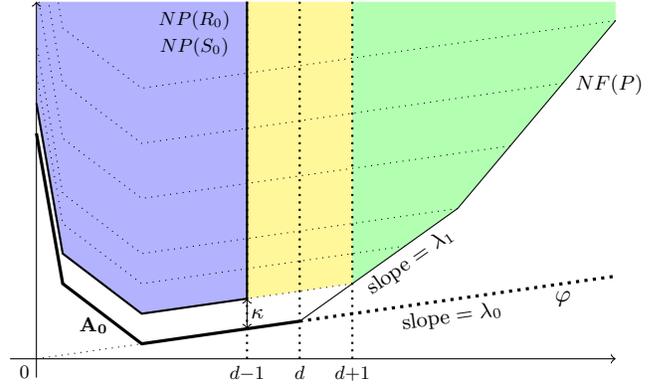
\begin{figure}
\hfill
\begin{tikzpicture}[xscale=0.7, yscale=0.5]
\fill[blue!30] (4,9.5)--(4,1.6)--(2,1.2)--(0.5,2.8)--(0,6.8)--(0,9.5)--cycle;
\fill[yellow!50] (6,9.5)--(6,2)--(4,1.6)--(4,9.5)--cycle;
\fill[green!30] (6,9.5)--(6,2)--(8,4)--(11,9)--(11,9.5)--cycle;
\draw[->] (-0.5,0)--(11,0);
\draw[->] (0,-0.5)--(0,9.5);
\draw[very thick] (0,6)--(0.5,2)--(2,.4)--(5,1);
\draw(5,1)--(8,4)--(11,9);
\draw[thick] (4,9.5)-- (4,1.6)--(2,1.2)--(0.5,2.8)--(0,6.8);
\draw[dotted] (6,2)--(4,1.6);
\draw[dotted] (7,3)--(2,2)--(0.5,3.6)--(0,7.6);
\draw[dotted] (8,4)--(2,2.8)--(0.5,4.4)--(0,8.4);
\draw[dotted] (9,5.667)--(2,4.267)--(0.5,5.867)--(0.046,9.5); 
\draw[dotted] (10,7.333)--(2,5.733)--(0.5,7.333)--(0.229,9.5); 
\draw[dotted] (11,9)--(2,7.2)--(0.5,8.8)-- (.413,9.5);
\draw[very thick,dotted] (5,1)--(11,2.2);
\draw[dotted] (0,0)-- (5,1);

\draw[<->] (4,.8)--(4,1.6);
\draw[dotted,thick] (5,0)--(5,9.5);
\draw[dotted,thick] (4,0)--(4,9.5);
\draw[dotted,thick] (6,0)--(6,9.5);
\node[scale=0.8, below left] at (0,0) { $0$ };
\node[scale=0.8, below] at (5,0) { $d$ };
\node[scale=0.8, below] at (4,0) { $d{-}1$ };
\node[scale=0.8, below] at (6,0) { $d{+}1$ };
\node[scale=0.85, below left] at (1.5,1.3) { $\mathbf{A_0}$ };
\node[scale=0.85, right] at (10.1,7.3) { $NF(P)$ };
\node[scale=0.9] at (4.2,1.2) { $\kappa$ };
\node[scale=0.8] at (3,9) { $NP(R_0)$ };
\node[scale=0.8] at (3,8.3) { $NP(S_0)$ };
\node[right,scale=0.9,rotate=35] at (6.2,1.7) { $\text{slope} = \lambda_1$ };
\node[right,scale=0.9,rotate=8] at (6.8,.9) { $\text{slope} = \lambda_0$ };
\node[thick] at (10,1.6) { $\varphi$ };
\end{tikzpicture}
\hfill\null

\vspace{-5mm}

\caption{Bound on $\NF(R_0)$ and $\NF(S_0)$}
\label{fig:boundR0S0}
\end{figure}
Finally observe that:
$$T_0 = (1 - B_0) \mod A_0 = \big((A_0 - P) \div A_0\big) \mod A_0.$$
Therefore $b_0(T_0) \geq \kappa$ results from $b_\varphi(A_0{-}P) \geq 
\kappa$ thanks to Lemma~\ref{lem:bphib0}. 
We have then established~\eqref{eq:induction} when $i = 0$.

We now assume~\eqref{eq:induction} for the index $i$.
From $A_{i+1} = A_i + R_i$ and the estimation $b_\varphi(R_i) \geq 2^i 
\kappa > 0$, we derive $\NF(A_{i+1}) = \varphi_{|[0,d]}$. Therefore,
Lemma~\ref{lem:bphib0} applies with $A = A_i$ and $A = A_{i+1}$.
Now coming back to the the definition of $B_{i+1}$ and using
Theorem~\ref{theo:EDivisionNP}, we get $b_0(B_{i+1}) \geq b_1(B_{i+1}) 
\geq 0$. As a consequence:
$$b_\varphi(R_i B_{i+1}) \geq b_\varphi(R_i) + b_0(B_{i+1})
\geq 2^i \kappa$$
by Lemma~\ref{lem:bphib0} and the induction hypothesis.
Using again Lemma~\ref{lem:bphib0}, we then derive from 
\eqref{eqdef:Biminus} and \eqref{eqdef:Siminus} that $b_0(\Delta B_i) \geq 
2^i \kappa$ and $b_\varphi(\Delta S_i) \geq 2^i \kappa$.
Similarly, using \eqref{eqdef:Si2} and the estimations we already know,
we obtain $b_\varphi(S_i) \geq 2^i \kappa$. Combining this with
$b_\varphi(\Delta S_i) \geq 2^i \kappa$, we find
$b_\varphi(S_{i+1}) \geq 2^i \kappa$ as well. Applying again and again
the same strategy, we deduce successively
$b_0(\Delta V_i) \geq 2^i \kappa$ using~\eqref{eqdef:Viminus},
$b_0(1{-}Q_i) \geq 2^i \kappa$ using~\eqref{eqdef:Qiminus},
$b_\varphi(R_{i+1}) \geq 2^{i+1} \kappa$ using~\eqref{eq:Riplus},
and then $b_0(T_{i+1}) \geq 2^{i+1} \kappa$
using~\eqref{eq:ViBiplus}. Finally, coming back to the recurrence
defining $V_{i+1}$ and remembering that $b_1(V_i)$ and $b_1(B_{i+1})$
are both nonnegative, we find $b_1(V_{i+1}) \geq 0$.
The equalities and inequalities
of Eq.~\eqref{eq:induction} have then all been established for
the index $i+1$ and the induction goes.

From the inequalities $b_\varphi(R_i) \geq 2^i \kappa$, we deduce that 
the sequence $(A_i)$ is Cauchy and therefore converges. Its limit 
$A_\infty$ certainly satisfies $\NF(A_\infty) = \varphi_{|[0,d]}$ because 
all the $A_i$'s do. Moreover we know that $b_\varphi(S_i) \geq 2^i 
\kappa$ from what we derive that the sequence $(S_i)$ goes to $0$.
Coming back to the definition of $S_i$, we find $P \mod A_\infty = 0$,
\emph{i.e.} $A_\infty$ divides $P$. Finally, Eq.~\eqref{eq:rateconv}
giving the rate of convergence follows from the writing
$A_\infty{-}A_i = \sum_{j=i}^\infty R_j$
together with the facts that $b_\varphi(R_j) \geq 2^i \kappa$ and $\deg 
R_j \leq d{-}1$ for all $j \geq i$.

\begin{rem}
\label{rem:slope-factor}
It follows from the proof above that the sequence $(V_i)_{i \geq 0}$
converges as well. Its limit $V_\infty$ is an inverse of $B_\infty
= P \div A_\infty$ modulo $A_\infty$ and it satisfies in addition
$b_1(V_\infty) \geq 0$.

Moreover, the conclusion of 
Theorem~\ref{theo:slope-factor} is still correct if $A_0$ is any 
polynomial of degree $d$ with leading coefficient $a_d$ and $V_0$
is any polynomial as soon as they satisfy:
$$b_\varphi\big(V_0 P \mod A_0\big) > 0
\quad \text{and} \quad
b_0\big((1 - V_0 B_0) \mod A_0\big) > 0$$
except that the constant $\kappa$ giving the rate of convergence
should be now
$\kappa = \min \big( 
b_\varphi\big(V_0 P \mod A_0\big), \,
b_0\big((1 - V_0 B_0) \mod A_0\big)\big)$.
\end{rem}

\subsection{A slope factorization algorithm}

Let $P \in K_n[X]$ and $d$ be the abscissa of an extremal point of 
$\NP(P)$. Previously (\emph{cf} Theorem~\ref{theo:slope-factor}), we 
have defined a sequence $(A_i, V_i)$ converging to $(A,V)$ where $A$ is 
a factor of $P$ whose Newton function is $\NF(P)_{|[0,d]}$ and $V$ is the inverse of $B = P/A$ modulo $A$.
We now assume that $P$ is known up to some finite precision: $P = P_\app + 
O(\hspace{0.5mm}\cdots)$ where the object inside the $O$ depends on the 
chosen precision model. We address the two following questions:
(1)~what is the precision on the factor $A$, and
(2)~how can one compute in practice $A$ at this precision?

In the sequel, it will be convenient to use a different normalization on 
$A$ and $B$: if $a_d$ is the coefficient of $P$ of degree $d$, we set
$A^{(1)} = a_d^{-1} A$ and $B^{(1)} = a_d B$
so that $A^{(1)}$ is monic and $P = A^{(1)} B^{(1)}$. We shall also 
always assume that $P$ is monic in the sense that its leading
coefficient is \emph{exactly}~$1$; the precision datum on $P$ then 
only concerns the coefficients up to degree $n{-}1$. Similarly, noting
that $A^{(1)}$ and $B^{(1)}$ are monic as well, they only carry a
precision datum up to degree $d{-}1$ and $n{-}d{-}1$ respectively.

\medskip

\noindent
{\bf Newton precision.}
We assume that the precision on the input $P$ has the shape 
$O(\varphi_P)$ where $\varphi_P$ is a Newton function of degree 
$n{-}1$. From now on, we assume that the precision $O(\varphi_P)$ 
is nondegenerate in the sense of Definition~\ref{def:nondeg}. This
ensures in particular that the Newton polygon of $P$ is well defined. 
We import the notations $\varphi$, $b_\varphi$ and $b_0$ from \S 
\ref{ssec:Newtoniter} and refer to 
Eqs.~\eqref{eqdef:functionphi}--\eqref{eqdef:b0} for the definitions.

\begin{prop}
\label{prop:Newtonprecslope}
We keep all the above notations and assumptions. We set:
$$\delta = \min_{x \in [0,n{-}1]} \varphi_P(x) - \varphi(x)$$
and assume that $\delta > 0$. 
Then the factor $A^{(1)}$ is known with precision at least 
$O(\varphi_{A^{(1)}})$ with
$\varphi_{A^{(1)}} = \varphi_{|[0,d{-}1]} - \varphi(d) + \delta$.
\end{prop}

\begin{proof}
Let $\delta_P \in K[X]$ be such that $\NP(\delta_P) \geq \varphi_P$.
Let $A_\app^{(1)}$ and $A^{(1)}$ be the monic factors of $P_\app$
and $P_\app + \delta P$ respectively whose Newton functions are
$\varphi_{|[0,d]} - \varphi(d)$.

We define the sequences $(A_i)$ and $(V_i)$ by the recurrence: 
\begin{align*} 
A_0 & = A_\app, \quad V_{0} = V_\app \medskip \\ A_{i+1} 
&= A_i + \big(V_i (P_\app + \delta_P) \:\mod\: A_i\big), \smallskip \\ 
V_{i+1} &= (2 V_i -V_i^2 B_{i+1} ) \mod A_{i+1} \\ 
&  \hspace{1cm}\text{where } B_{i+1} = (P_\app + \delta_P) \div A_{i+1}. 
\end{align*} 
where $A_\app$ and $V_\app$ are those related to $P_\app$. Note that 
$A_\app = a_d \cdot A^{(1)}_\app$ if $a_d$ denotes the coefficient of 
$X^d$ in $P_\app$. By Remark~\ref{rem:slope-factor}, we know that the 
sequence $(A_i)$ converges to $A = a_d \cdot A^{(1)}$ and furthermore:
$$\NP(A - A_\app) \geq \varphi_{|[0,d{-}1]} + 
b_\varphi\big((V_\app \cdot \delta P) \mod A_\app\big)$$ 
since $(V_\app P_\app) \mod A_\app = (1 - V_\app B_\app) \mod A_\app = 
0$. Using repeatedly Lemma~\ref{lem:bphib0}, we obtain:
$$b_\varphi\big((V_\app \cdot \delta P) \mod A_\app\big)
  \geq b_0(V_\app) + b_\varphi(\delta P) \\
  \geq b_\varphi(\delta P) \geq \delta.$$
Thus $\NP(A - A_\app) \geq \varphi_{[0,d{-}1]} + \delta$. Dividing
by $a_d$, we find
$\NP(A^{(1)} - A^{(1)}_\app) \geq \varphi_{A^{(1)}}$ and we are done.
\end{proof}

\begin{rem}
Under the hypothesis~\textbf{(H)} introduced below, a correct precision 
on $A^{(1)}$ is also $O(\psi_{A^{(1)}})$ where:
$$\psi_{A^{(1)}} = \big(\varphi_P \nftimes \big(\NF(V_\app)
- \NF(P)(d)\big)\big) \nfmod \NF(P)_{|[0,d]}.$$
This follows from Proposition~\ref{prop:precA1} using
$V^{(1)}_\app = a_d^{-1} V_\app$. It follows in addition from 
Remark~\ref{rem:slope-factor} that $\NF(V_\app)$ is bounded from below 
by $x \mapsto \lambda_1 x$. This yields the bound
$\psi_{A^{(1)}} \geq \big(\varphi_P \nftimes \psi\big) \nfmod 
\NF(P)_{|[0,d]}$
where $\psi : [0,d{-}1] \to \R$ is the affine function mapping $x$
to $\lambda_1 x - \NF(P)(d)$. 
\end{rem}

We can now move to the second question we have raised before, 
\emph{i.e.} the design of an algorithm for computing $A^{(1)}$ with the 
precision given by Proposition~\ref{prop:Newtonprecslope}. Our strategy 
consists in computing first the precision and applying then the Newton 
iteration until the expected precision is reached. Below is the
precise description of our algorithm.

\noindent\hrulefill

\noindent {\bf Algorithm} {\tt slope\_factorisation\_Newton}

\noindent{\bf Input:} a monic polynomial $P + O(\varphi_P) \in K_n[X]$,

\noindent\phantom{{\bf Input:} }a break point $d$ of $\NP(P)$

\noindent{\bf Output:} the factor $A$ described above

\smallskip\noindent 1.\ %
Compute the functions $\NF(P)$ and $\varphi$

\smallskip\noindent 2.\ %
Compute $\varphi_A = \varphi_{|[0,d-1]} - \varphi(d) + \min_{x \in [0,n]}
\varphi_P(x) - \varphi(x)$

\smallskip\noindent 3.\ %
Compute $\kappa = \NP(P)(d{+}1) + \NF(P)(d{-}1) - 2 \cdot \NF(P)(d)$

\smallskip\noindent 4.\ %
Set $i = 0$, $A_0 = \sum_{i=0}^d a_i X^i$
($a_i = \text{coeffs of } P$), $V_0 = 1$

\smallskip\noindent 5.\ %
{\bf repeat until} $\varphi_{|[0,d{-}1]} + 2^i \kappa \geq \varphi_A$

\smallskip\noindent 6.\ \hspace{5mm}%
lift $A_i$, $V_i$ and $P$ at enough precision

\smallskip\noindent 7.\ \hspace{5mm}%
compute 

\smallskip

\hspace{1cm}$\bullet$
$A_{i+1} = A_i + (P V_i \mod A_i)$

\hspace{1cm}\phantom{$\bullet$ }%
at precision $O(\varphi_{|[0,d{-}1]} + 2^{i+1} \kappa)$

\smallskip

\hspace{1cm}$\bullet$
$V_{i+1} = (2 V_i - V_i^2 \cdot (P \div A_{i+1})) \mod A_{i+1}$

\hspace{1cm}\phantom{$\bullet$ }%
at precision $O(x \mapsto \lambda_1 x + 2^{i+1}\kappa)$

\smallskip\noindent 8.\ \hspace{5mm}%
set $i = i+1$

\smallskip\noindent 9.\ 
{\bf return} $A_i + O(\varphi_A)$

\vspace{-1ex}\noindent\hrulefill

\begin{rem}
The precision needed at line 6 is of course governed by the computation
performed at line 7. Note that it can be either computed \emph{a
priori} by using Proposition~\ref{prop:NewtonprecEuclide} or dynamically 
by using relaxed algorithms from 
\cite{hoeven:02a,hoeven:07a,berthomieu-hoeven-lecerf:11a}. In both 
cases, it is in $O(\pi^{N_i})$ with $N_i = O(2^i \kappa + \min 
\NF(P))$.
\end{rem}

\noindent
It follows from Theorem~\ref{theo:slope-factor}, 
Remark~\ref{rem:slope-factor} and Proposition~\ref{prop:Newtonprecslope} 
that Algorithm {\tt slope\_factorisation\_Newton} is correct and 
stable. Using the standard soft-$O$ notation $\softO(\cdot)$ for
hiding logarithmic factor, our algorithm performs at most $\softO(n)$ 
combinatorial operations and $\softO(n)$ operations in $K$ at precision 
$O(\pi^N)$ with $N = O(\max \varphi_P - \min \NF(P))$ if one uses 
quasi-optimal algorithms for multiplication and Euclidean division of 
polynomials.

\medskip

\noindent
{\bf Lattice precision.}
The precision datum is given here by a lattice $H_P$ in $K_{\leq n{-}1}
[X]$; we shall then write
$P = P_\app + O(H_P)$
where $P_\app$ is a \emph{monic} approximation of the inexact polynomial 
$P$ we want to factor. We assume from now that $H_P$ is sufficiently 
small so that the Newton polygon of $P$ is well defined. We then can
define the function:
$$F = (F_A, F_B) : P_\app + H_P \to K_{=d}[X] \times K_{=n{-}d}[X]$$
mapping a polynomial $P$ to the couple $(A^{(1)}, B^{(1)})$ obtained
from it. We set $(A_\app^{(1)}, B_\app^{(1)}) = F(P_\app)$.

We make the following hypothesis \textbf{(H)}:

\medskip

\noindent
\hfill
\begin{minipage}{8cm}
The lattice $H_P$ is a first order lattice at every point
of $P_\app + H_P$ in the sense of \cite[Definition~3.3]{caruso-roe-vaccon:14a},
\emph{i.e.} for all $P \in P_\app + H_P$:

\vspace{-5mm}

$$F(P + H_P) = F(P) + F'(P)(H_P).$$
\end{minipage}
\hfill \null

\bigskip

\noindent
Obviously \textbf{(H)} gives an answer to the first question we have 
raised above: the precision on the couple $A^{(1)}$ is the lattice 
$H_{A^{(1)}}$ defined as projection on the first component of 
$F'(P_\app)(H_P)$. It turns out that it can be computed explicitely as 
shown by the next proposition.

\begin{prop}
\label{prop:precA1}
The application $F_A : P \mapsto A^{(1)}$ is of class $C^1$ on $P_\app + H_P$
and its differential at some point $P$ is the linear mapping
$$dP \mapsto dA^{(1)} = (V^{(1)} \: dP) \mod A^{(1)}$$
where $(A^{(1)}, B^{(1)}) = F(P)$ and $V^{(1)}$ is the inverse of 
$B^{(1)}$ modulo $A^{(1)}$.
\end{prop}

\begin{proof}
The function $F$ is injective and a left inverse of it is
$G : (A^{(1)},B^{(1)}) \mapsto A^{(1)}B^{(1)}$.
Clearly $G$ is of class $C^1$ and its differential is given by
\begin{equation}
\label{eq:dPdAdB}
(dA^{(1)}, dB^{(1)}) \mapsto dP = A^{(1)} \: dB^{(1)} + B^{(1)} \: dA^{(1)}.
\end{equation}
Thanks to Bézout Theorem, it is invertible as soon as $A^{(1)}$ and 
$B^{(1)}$ are coprime, which is true because $\NP(A^{(1)})$ and 
$\NP(B^{(1)})$ do not shape a common slope. As a consequence $F$ is of 
class $C^1$ and its differential is obtained by inverting 
Eq.~\eqref{eq:dPdAdB}. Reducing modulo $A^{(1)}$, we get $dP \equiv 
B^{(1)} \: dA^{(1)} \pmod {A^{(1)}}$. The claimed result follows after 
having noticed that $dA^{(1)}$ has degree at most $d{-}1$.
\end{proof}

A remarkable Corollary of Proposition~\ref{prop:precA1} asserts the
optimality of Proposition~\ref{prop:Newtonprecslope} in a particular 
case.

\begin{cor}
We assume~\textbf{(H)}. Let $\delta \in \R$.
When the precision of $P$ is given by $O(\NF(P)_{|[0,d-1]} + \delta)$,
the precision of $A^{(1)}$ given by 
Proposition~\ref{prop:Newtonprecslope} is optimal.
\end{cor}

\begin{proof}
First remark that the $\delta$ defined in the statement of 
Proposition~\ref{prop:Newtonprecslope} coincide with the $\delta$
introduced in the Corollary. Define
$\varphi_{A^{(1)}} = \varphi_{|[0,d{-}1]} - \varphi(d) + \delta$.
Let $H_P$ (resp. $H_{A^{(1)}})$ be the lattice consisting of
polynomials of degree at most $n{-}1$ (resp. at most $d{-}1$)
whose Newton function is not less that $\varphi_P = \NF(P)_{|[0,d-1]} + \delta$ (resp. 
$\varphi_{A^{(1)}}$). We have to show that $F'_A(P_\app)(H_P) = 
H_{A^{(1)}}$. According to Proposition~\ref{prop:precA1}
the mapping $G : K_{\leq d{-}1}[X] \to K_{\leq n{-}1}[X]$,
$dA^{(1)} \mapsto (B_\app^{(1)} \: dA^{(1)}) \mod A_\app^{(1)}$
is a right inverse of $F'_A(P_\app)$. It is then enough to prove
that $G$ takes $H_{A^{(1)}}$ to $H_P$, which can be done easily
using Theorem~\ref{theo:EDivisionNP}.
\end{proof}

As for the second question, the discussion is similar to the case of 
Newton precision expect that we need a new stopping criterion. It is given 
by the following proposition.

\begin{prop}
\label{prop:stop}
We assume \textbf{(H)}. 

\vspace{-1mm}

\begin{enumerate}[(i)]
\item Let $\tilde A^{(1)} \in K_{=d}[X]$ such that
$\tilde A^{(1)} \tilde B^{(1)} \in P_\app + O(H_P)$ with
$\tilde B^{(1)} = P_\app \div \tilde A^{(1)}$.
Then $\tilde A^{(1)} \in A_\app^{(1)} + H_{A^{(1)}}$.
\vspace{-1mm}
\item Let in addition $\tilde V^{(1)} \in K_{\leq d{-}1}[X]$ such that:
$$\big(\tilde B^{(1)} \tilde V^{(1)} \cdot H_P\big) \mod \tilde A^{(1)}
= H_P \mod \tilde A^{(1)}.$$
Then $H_{A^{(1)}} = \big(\tilde V^{(1)} \cdot H_P\big) \mod \tilde A^{(1)}$.
\end{enumerate}
\end{prop}

\begin{proof}
\emph{(i)} 
Set $\tilde P = \tilde A^{(1)} \tilde B^{(1)}$. We know by assumption
that $\tilde P = P_\app + O(H_P)$. Thus $F(\tilde P)$ is well defined.
The unicity of the slope factorization (\emph{cf} 
Remark~\ref{rem:unicity}) further implies that $F(\tilde P) = 
(\tilde A^{(1)}, \tilde B^{(1)})$. The claimed result now follows from
the hypothesis~\textbf{(H)}.

\emph{(ii)} By applying~\textbf{(H)} with $P = \tilde P$ and replacing 
$F'(\tilde P)$ by its expression given by Proposition~\ref{prop:precA1}, 
we find:
$$H_{A^{(1)}} = \big((\tilde B^{(1)})^{-1} \cdot H_P\big) \mod \tilde A^{(1)}.$$
Dividing $\big(\tilde B^{(1)} \tilde V^{(1)} H_P\big) \mod \tilde 
A^{(1)} = H_P \mod \tilde A^{(1)}$ by $\tilde B^{(1)}$ modulo $\tilde
A^{(1)}$, we get
$H_{A^{(1)}} = \big(\tilde V^{(1)} \cdot H_P\big) \mod \tilde A^{(1)}$
as expected.
\end{proof}

As a conclusion, the algorithm we propose consists in computing the 
Newton sequences $(A_i)$ and $(V_i)$ (following the strategy of the
algorithm \texttt{slope\_factorisation\_Newton} regarding to precision)
until we find a couple $(\tilde A^{(1)}, \tilde V^{(1)})$ satisfying 
the requirements (i) and (ii) of Proposition~\ref{prop:stop}. Once this 
couple has been found, one may safely output $\tilde A^{(1)} + 
O\big(\big(\tilde V^{(1)} \cdot H_P\big) \mod \tilde A^{(1)}\big)$ under
\textbf{(H)}.
The resulting algorithm has quasi-optimal running 
time and optimal stability.

\bibliographystyle{plain}


\end{document}